\documentclass[a4paper,12pt]{article}
\usepackage[utf8]{inputenc}
\usepackage{amssymb}
\usepackage{amsmath}
\usepackage{amsthm}
\usepackage{MnSymbol} 
\usepackage{cancel}
\usepackage{ulem}
\usepackage{mathtools}
\usepackage{stmaryrd}
\usepackage{blkarray} 
\usepackage{tikz}
\usetikzlibrary{arrows}
\usetikzlibrary{shapes,decorations}
\usepackage{layouts}

\usepackage{todonotes}

\usepackage{tikz-cd}
\usepackage{fancyhdr}
\usepackage{authblk}

\usepackage{xcolor}

\usepackage{hyperref}

\usepackage{caption}
\usepackage{subcaption}

\newtheorem{definition}{Definition}

\newtheorem{theorem}{Theorem}
\newtheorem{lemma}{Lemma}
\newtheorem{corollary}{Corollary}

\DeclareMathOperator{\Ima}{Im} 
\DeclareMathOperator{\Hom}{Hom} 

\newcommand{\KP}[1]{%
  \begin{tikzpicture}[baseline=-\dimexpr\fontdimen22\textfont2\relax]
  #1
  \end{tikzpicture}%
}

\newcommand{\KPB}{%
  \KP{
    \draw[color=gray,thick] (-0.3,0.3) -- (0.3,-0.3);
    \draw[color=gray,thick] (-0.3,-0.3) -- (-0.05,-0.05);
    \draw[color=gray,thick] (0.05,0.05) -- (0.3,0.3);
  }%
}
\newcommand{\KPC}{%
  \KP{%
    \draw[color=gray,thick] (-0.3,0.3) .. controls (0,-0.05) .. (0.3,0.3);
    \draw[color=gray,thick] (-0.3,-0.3) .. controls (0,0.05) .. (0.3,-0.3);
  }%
}
\newcommand{\KPD}{%
  \KP{%
    \draw[color=gray,thick] (-0.3,-0.3) .. controls (0.05,0) .. (-0.3,0.3);
    \draw[color=gray,thick] (0.3,-0.3) .. controls (-0.05,0) .. (0.3,0.3);
  }%
}

\author[1]{Benjamin Jones \thanks{jones657@msu.edu}}
\author[1,2,3]{Guo-Wei Wei \thanks{weig@msu.edu}}
\affil[1]{Department of Mathematics, Michigan State University, MI, 48824, USA}
\affil[2]{Department of Electrical and Computer Engineering, Michigan State University, MI 48824, USA}
\affil[3]{Department of Biochemistry and Molecular Biology, Michigan State University, MI 48824, USA}

\makeatletter
    \renewcommand*{\@fnsymbol}[1]{\ensuremath{\ifcase#1\or \dagger\or *\or *\or
   \mathsection\or \else\@ctrerr\fi}}
\makeatother
\date{}

\title{Khovanov Laplacian and Khovanov Dirac for Knots and Links}
\date{December 2024}

\begin{document}
\maketitle
\paragraph{Abstract} 
Khovanov homology has been the subject of much study in knot theory and low dimensional topology since 2000. 
This work introduces a  Khovanov Laplacian and a  Khovanov  Dirac to study knot and link diagrams. The harmonic spectrum of the Khovanov Laplacian or the  Khovanov  Dirac retains the topological invariants of  Khovanov homology, while their non-harmonic spectra reveal additional information that is distinct from  Khovanov homology. 



\newpage
    \tableofcontents
\newpage

\section{Introduction}
Knots are ubiquitous in nature.  Biomolecules, such as DNA, RNA, and proteins, have been shown to contain knots and pseudoknots. For example, the two twisted strands of DNA can form a closed link in mitochondrial DNA and bacterial chromosomes, and the linking number is related to how tightly packed the DNA can be, its stability under high temperatures, and the speed of DNA replication and transcription, all of which have important physiological implications \cite{arsuaga2005dna}. The topology of knots and local knotted structures, such as slipknots, in proteins have been studied in \cite{shen_knot_2023}, which introduced a multiscale Gauss link integral (mGLI) and the general paradigm of knot data analysis (KDA)  for biological data. Knots, slipknots, and knotoids have been studied in proteins by the Knoto-ID and KnotProt projects \cite{Dorier_KnotoID, Jamroz_KnotProt}. 

Mathematically, knot theory is a branch of geometric topology that studies the embeddings of twisted and/or interlaced   circles into the 3-space. An important issue in the knot theory is whether two knots are equivalent, i.e., whether they can be transformed into each other through many smooth deformations. Knot invariants are a set of properties that remain unchanged under knot transformations. Commonly used knot invariants include knot polynomials, knot groups, and hyperbolic invariants. Many tools are developed to study knots, as well as their representations via planar diagrams. 

As a subject in geometric topology, knots were studied with the aforementioned  polynomials, groups, and diagrams. In 2000, Khovanov introduced a categorification  of the Jones polynomial \cite{khovanov_categorification_2000}, a celebrated polynomial link invariant that connected the fields of operator algebras, algebraic topology,  and geometric topology \cite{JonesPolynomial}. Khovanov homology is a bigraded cohomology theory for link projections that recovers the Jones polynomial from its graded Euler characteristic. Moreover, there are pairs of knots with identical Jones polynomial, but distinct Khovanov homology \cite{Bar_Natan_2002}, so the Khovanov homology contains strictly more information. It was shown by Kronheimer and Mrowka in \cite{kronheimer_khovanov_2011} that the unknot is the unique knot with Khovanov homology of rank $1$, while it is not yet known if the Jones polynomial detects the unknot. Various authors have shown that other links and knots are uniquely detected by Khovanov homology, such as the unlinks by Hedden and Ni \cite{hedden_khovanov_2013}, Hopf links by Baldwin et al. \cite{Baldwin2018KhovanovHopf}, the trefoils by Baldwin and Sivek \cite{BaldwinSivek2022trefoil}, and the figure-eight knot by Baldwin, et  al \cite{Baldwin2021figure8}. Since Khovanov's original categorification, there have been a number of related invariants developed, such as Eun-Soo Lee's homology \cite{lee_endomorphism_2005}, Rasmussen's $s$ invariant \cite{rasmussen_khovanov_2010}, Ozsv\'ath et al.'s odd Khovanov homology \cite{ozsvath_odd_2013}, and various other quantum link invariants.  

However, Khovanov homology has been rarely exploited for practical applications to data science. In contrast, persistent homology  tracks multiscale topological structure of point cloud data via a filtration process  and is the key technique for topological data analysis (TDA)   \cite{zomorodian2004computing}. TDA  has been widely applied various fields, including the visualization of relationships between knot invariants \cite{MapperOnBallMapper}. However,  persistent homology   is a coarse measurement in the sense that it ignores structures that are equivalent up to homotopy. This necessitates the development of more powerful topological techniques such as persistent Laplacians defined on point cloud data \cite{wang2020persistent}. The foundation that a persistent   Laplacian is built upon is a combinatorial Laplacian, which generalizes the graph Laplacian, a central object of spectral graph theory, to simplicial complexes \cite{Eckmann1944/45}, which can encode higher-dimensional data. Mathematical analysis of persistent   Laplacian has been reported  \cite{memoli2022persistent,liu2023algebraic}.
Persistent combinatorial Laplacians  have been extended to many topological structures, such as path complexes, directed flag complexes, directed hypergraphs, and cellular sheaves as surveyed in \cite{wei2023persistent}. These formulations, together with persistent Hodge Laplacians defined on differentiable manifolds \cite{chen2021evolutionary}, are called persistent topological Laplacians  \cite{wei2023persistent}.  
A key property of persistent topological Laplacians is that they share the same algebraic structure but may be defined on different topological spaces or data forms, such as simplicial complexes \cite{wang2020persistent}, path complexes, directed flag complexes, directed hypergraphs,  cellular sheaves, differentiable manifolds \cite{chen2021evolutionary}, and curves (including knots and links)\cite{Shen_2024}. Additionally,  combinatorial  Laplacians defined  capture strictly more information than the corresponding homology, which is studied through the non-harmonic  spectra (non-zero eigenvalues) of the Laplacian. A similar relationship also exists between the de Rham cohomology and the Hodge Laplacian defined on differentiable manifolds \cite{dodziuk1977rham,wei2023persistent}. A generalization has been given to the Mayer homology and Mayer Laplacians \cite{shen2024persistent}.   This relationship between (co)homology and topological Laplacians motivates the present study of Khovanov Laplacian through Khovanov homology defined on knots and/or links. 

An alternative perspective is  the  graph Laplacian of knots. Knots, and their multiple-component generalization, links, are often studied through planar diagrams, a 2D projection of the 3D knot or link. Under certain restrictions of the planar diagram and in combination with an unshaded checkerboard graph for the diagram, Silver and Williams constructed weighted graph Laplacians that recover a Seifert matrix, a Goeritz matrix, or an Alexander matrix for the link \cite{silver_knot_2019}. From these, they recovered the Alexander polynomial and $\omega$-signature, two important knot invariants. Their work establishes a connection between a graph Laplacian and key properties of knots. Since the extension of graph Laplacians on graphs to simplicial complexes lead to combinatorial Laplacians \cite{Eckmann1944/45}, the extension of graph Laplacians that recover knot polynomials to their corresponding knot homology theory should also give rise to knot Laplacians, such as the Khovanov Laplacian, which serves as a generalization of Khovanov homology.


Parallel to the success of persistent Laplacians \cite{wang2020persistent, wei2023persistent, meng2021persistent}, topological Dirac has become a popular topic in data science  \cite{ameneyro2024quantum, wee2023persistent, suwayyid2023persistent}. Dirac operator has various applications in physics, materials\cite{wehling2014dirac}, and dynamics \cite{carletti2024global}. Mathematically, the topological Dirac is closely related to topological Laplacian as both are defined through the boundary operator or   coboundary operator \cite{wee2023persistent,wei2023persistent}. As such, we formulate a Khovanov  Dirac for knots and links in this work as well. 
  
The rest of this paper is organized as follows. In Section \ref{sec:background} we give a brief introduction to the standard notions of combinatorial Laplacian and Dirac and Khovanov homology. In Section \ref{sec:khovanovLaplacian}, we construct the   Khovanov Laplacian and   Khovanov Dirac and analyze some of their properties. A conclusion is given in Section \ref{sec:conclusion}. Tables of the Khovanov Laplacian spectra for prime knots are contained in the supporting information. 

\section{Background}\label{sec:background}

\subsection{Combinatorial Laplacian and Dirac}\label{subsec:laplacian}
Here we briefly describe the combinatorial Laplacian and Dirac defined for an unweighted simplicial complex. The algebra structure is identical to what we will construct for the Khovanov Laplacian, but this lays the groundwork in how we will approach it while avoiding the complexities of the bigraded Khovanov cochain complex. This is an accelerated construction without much prose. More detailed treatments on combinatorial Laplacians for simplicial complexes are contained in literature \cite{horak2013spectra, LimHodge2020}.

Consider a finite simplicial complex $K$ with the set of $n$-cells $K_n$, chains $C_n=\mathrm{Span}_\mathbb{R}\{K_n\}$, and cochains $C^n=\hom(C_n,\mathbb{R})$. Writing a generic $n$-simplex as $[0,1,\dots,n]$, we use the standard boundary map $d_n:C_n\to C_{n-1}$ given by 

\[
    d_n([0,1,\dots,n]) = \sum_{i=0}^n (-1)^i [0,1,\dots,i-1,i+1,\dots,n],
\]

which satisfies $d_{n-1}\circ d_{n}=0$. By constructing the chains with real coefficients, we can equip $C_n$ with an inner product $\langle\cdot,\cdot\rangle:K_n\times K_n\to \mathbb{R}$ by defining it on basis elements via  

\[
    \langle\sigma_i,\sigma_j\rangle=\begin{cases}
        1 &\text{if }i=j,\\
        0 &\text{otherwise}
    \end{cases}
\]

and extending linearly. We may then identify $C_n$ with its dual. This induces an inner product $\llangle\cdot,\cdot\rrangle:C^n\times C^n\to\mathbb{R}$ via

\[
    \llangle f,g\rrangle = \sum_{\sigma\in C_n}f(\sigma)g(\sigma).
\]

With respect to these inner products, there is a unique adjoint of the boundary map, $d_n^*:C_{n-1}\to C_n$. We will need to explicitly refer to the matrix forms of $d_n$ and $d_n^*$, which are $B_n$ and $B_n^T$, respectively. We may now define the combinatorial Dirac and Laplacian.

\begin{definition}
    The $\boldsymbol{n}$\textbf{-th combinatorial Dirac} $\mathcal{D}_n:\bigoplus_{i\leq n+1} C_i\to \bigoplus_{i\leq n+1} C_i$ is given by
\[
    \mathcal{D}_n = \begin{bmatrix}\mathbf{0}_{m_0\times m_0} & B_1 & \mathbf{0}_{m_0\times m_2} & \cdots &\mathbf{0}_{m_0\times m_{n}} & \mathbf{0}_{m_0\times m_{n+1}}\\
    											 B_1^T & \mathbf{0}_{m_1\times m_1} & B_2 & \cdots &\mathbf{0}_{m_1\times m_{n}} &\mathbf{0}_{m_1\times m_{n+1}}\\
    											 \mathbf{0}_{m_2\times m_0}& B_2^T & \mathbf{0}_{m_2\times m_2} & \cdots & \mathbf{0}_{m_2\times m_{n}} & \mathbf{0}_{m_2\times m_{n+1}}\\
        \vdots & \vdots & \vdots & \ddots &\vdots & \vdots\\
        
        \mathbf{0}_{m_{n}\times m_0} & \mathbf{0}_{m_{n}\times m_1} & \mathbf{0}_{m_{n}\times m_2} &\cdots & \mathbf{0}_{m_n\times m_{n}} & B_{n+1}\\
         \mathbf{0}_{m_{n+1}\times m_0} & \mathbf{0}_{m_{n+1}\times m_1} & \mathbf{0}_{m_{n+1}\times m_2} &\cdots &  B_{n+1}^T & \mathbf{0}_{m_{n+1}\times m_{n+1}}\\
    \end{bmatrix},
\]

	where $m_i=\dim C_i$. 
\end{definition}

Here, $\mathcal{D}_n$ is a block matrix where the block super-diagonal consists of the boundary matrices up to $B_{n+1}$, and then the matrix is symmetrized. There are two conventions for defining the discrete Dirac operator for simplicial complexes. We have chosen the one presented in \cite{suwayyid2023persistent,wee2023persistent}, although the convention used in \cite{ameneyro2024quantum, Bianconi_2021} could also be used. If we were to translate that convention to this context, we would have a similar Dirac operator built from only $d_n$, $d_{n+1}$, and their adjoints.

\begin{definition}
    The $\boldsymbol{n}$\textbf{-th combinatorial Laplacian} $\Delta_n:C_n\to C_n$ is given by
\[
    \Delta_n = d_{n+1}\circ d_{n+1}^* + d_n^* \circ d_n.
\]

	We sometimes separately consider $\Delta_n^{\text{up}}= d_{n+1}\circ d_{n+1}^*$ and $\Delta_n^{\text{down}} = d_n^*\circ d_n$, so that $\Delta_n=\Delta_n^{\text{up}}+\Delta_n^{\text{down}}.$
\end{definition}

We take $\Delta_0=d_1 \circ d_1^*$ and for $N$ the top dimension of $K$ we take $\Delta_N=  d_N^*\circ d_N$. The combinatorial Laplacian has the following key property, first shown by Eckmann \cite{Eckmann1944/45}. 

\begin{theorem}
\[
    \ker\Delta_n \cong H_n(K; \mathbb{R})
\] 
\end{theorem}

For more details on this result, see \cite{LimHodge2020}, particularly Theorems 5.1, 5.2, and 5.3.

The image of $\Delta_n$ contains additional information about $K$. It is common to analyze $\Delta_n$ via its eigenvalues, which are all real and nonnegative because $\Delta_n$ is self-adjoint and positive semi-definite. We may order them in nondecreasing order $S_n=\{\lambda_0,\lambda_1,\dots,\lambda_M\},$ where $M=\dim C_n$. 

\begin{definition}
    We refer to the $0$-eigenvalues of $\Delta_n$ as the \textbf{harmonic spectra}, the nonzero eigenvalues as \textbf{non-harmonic} spectra, and the full multiset $S_n$ of eigenvalues as the \textbf{spectra}.
\end{definition}

The multiplicity of zero as an eigenvalue is the dimension of the null space of $\Delta_n$, and hence the dimension of $H_n(K)$, which is the $n$-th Betti number $\beta_n$. It is common to distinguish the least nonzero eigenvalue, denoted $\lambda$, because in the special cases where $\Delta_n$ coincides with a graph Laplacian (i.e., the  0-th combinatorial Laplacian), this $\lambda$ is called the Fiedler value or spectral gap and communicates information about the graph such as algebraic connectivity. 

In differential geometry, the Hodge Dirac operator squares to the Hodge Laplacian operator. In this setting, we have a similar relationship: $\mathcal{D}_n^2$ is a block diagonal matrix where the diagonal is composed of $\Delta_0,\Delta_1,\dots, \Delta_n,\Delta_{n+1}^{\text{down}}$.

We note that the terms ``spectrum," ``spectra," and ``spectral" are widely used for other concepts in algebraic and geometric topology. In this work we exclusively refer to spectra as the eigenvalues of self-adjoint linear operators, and not to any of the spectral sequences related to Khovanov homology.

\subsection{Khovanov Homology}

Here we introduce the notions of   knots, links, and Khovanov homology. Further information on the standard notions of knot theory is in \cite{lickorish_introduction_1997}. For Khovanov homology, the interested reader is referred to an excellent introduction by Dror Bar-Natan in  \cite{Bar_Natan_2002}. Khovanov homology is a bigraded homology theory for links, and the graded Euler characteristic yields the Jones polynomial. We lay out essentially the same definitions and notations as \cite{Bar_Natan_2002}.

A \textbf{link} $L$ is an embedding of finitely many copies of $S^1$ into $\mathbb{R}^3$ or $S^3$, and a \textbf{knot} is a link with only $1$ component.  Two links are \textbf{equivalent} if there is an ambient isotopy between their embeddings. For any embeddings $f,g: N\to M$ of manifolds, an \textbf{ambient isotopy} from $f$ to $g$ is a family of maps $F_t:M\to M$ for $t\in[0,1]$ such that $F_0$ is the identity, $F_1\circ f=g$, and $F_t$ is a homeomorphism for all $t$. Links are often studied by their projections onto $\mathbb{R}^2$. When such a projection is injective except at finitely many points and where intersections are of only two segments of $L$, we can keep track of which segment is ``over" and which is ``under." A projection following these rules is called a \textbf{link diagram} or \textbf{knot diagram}. 

Two link diagrams are equivalent whenever there is a finite sequence of transformations called \textbf{Reidemeister moves} that transform one diagram to another. There are three Reidemeister moves, depicted in Figure \ref{fig:reidmeister}, which are often described as twisting and untwisting (Reidemeister 1 or R1), poking (Reidemeister 2 or R2), and sliding under or over a crossing (Reidemeister 3 or R3). Two link diagrams are equivalent by a sequence of Reidemeister moves if and only if the corresponding links are equivalent by ambient isotopy. Because they characterize equivalence of links, link diagrams are a common tool for displaying and computing properties of links. Many interesting properties are called \textbf{link invariants} (or \textbf{knot invariants}), which are structures and numbers that can be computed from a diagram, and remain unchanged by the Reidemeister moves. Prominent examples of link or knot invariants include the Jones polynomial \cite{JonesPolynomial}, crossing number \cite{lickorish_introduction_1997}, Alexander polynomial \cite{alexander_topological_1928}, and Khovanov homology \cite{khovanov_categorification_2000}.

\begin{figure}[htpb]
    \centering
    \begin{subfigure}{0.2\textwidth}
        \includegraphics[width=0.8\textwidth]{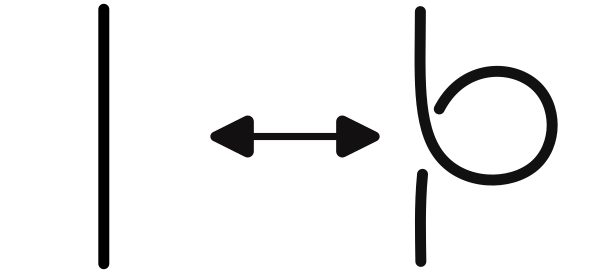}
        \caption{R1}
    \end{subfigure}
    \begin{subfigure}{0.2\textwidth}
        \includegraphics[width=0.8\textwidth]{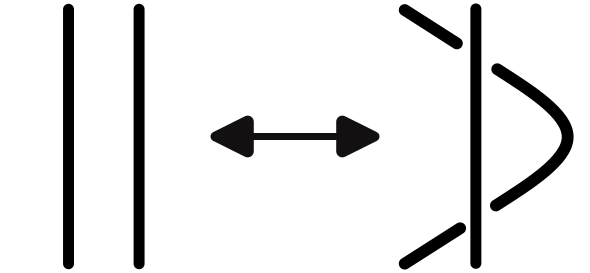}
        \caption{R2}
    \end{subfigure}
    \begin{subfigure}{0.2\textwidth}
        \includegraphics[width=0.8\textwidth]{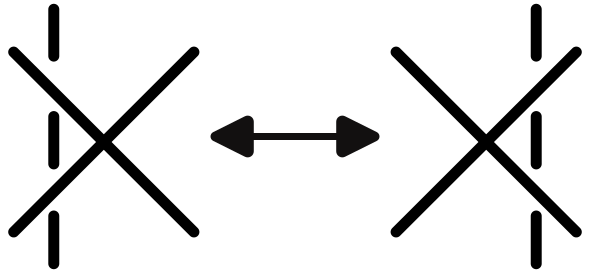}
        \caption{R3}
    \end{subfigure}
    \caption{The three Reidemeister moves.}
    \label{fig:reidmeister}
\end{figure}

For a recurring example, take the projection of the right-handed trefoil knot in Figure \ref{fig:trefoil}(a). We will use a Planar Diagram (PD) notation for the link projection that leads to a well-defined system for labeling cycles in a smoothing and ultimately the basis elements for our chain complex. This will allow us to easily compute a matrix form of the boundary map, which we will need for the Laplacian. 

For an oriented knot diagram, label the edges with increasing natural numbers following the orientation. For each component in a link diagram, do the same process, but with disjoint sets of natural numbers for each component. From this, one can label the crossings by  $X_{ijkl}$, where $i, j, k$, and $l$ the four edges that meet at the crossing, starting with the edge $i$ that is about to go under and listing the rest in counterclockwise order. In Figure \ref{fig:trefoil}(b) only the edges are labeled, and then in Figure \ref{fig:trefoil}(c) both the edges and the crossings are labeled. 

\begin{figure}[htpb]
    \centering
    \begin{subfigure}[t]{0.2\textwidth}
        \includegraphics[width=0.8\textwidth]{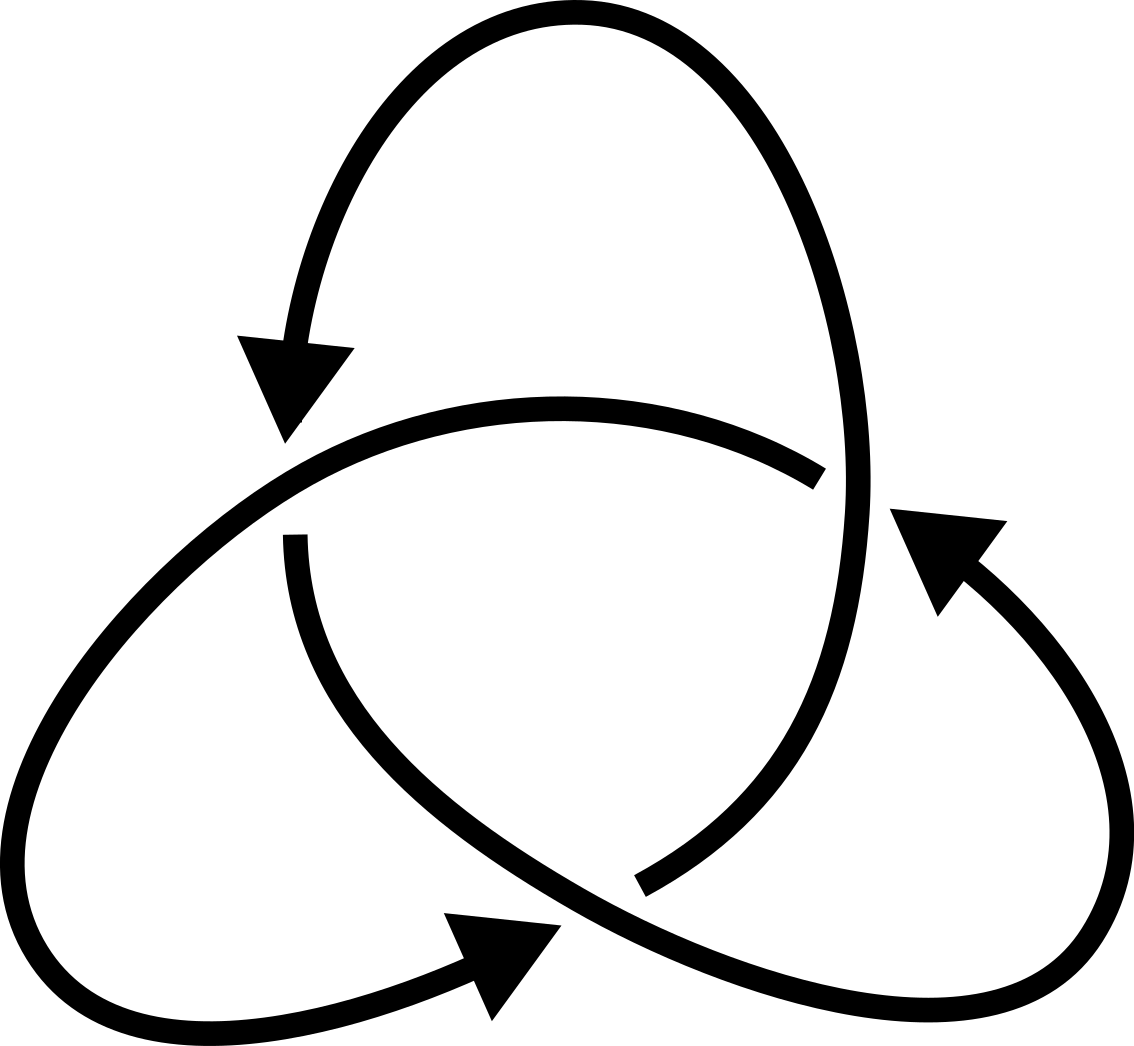}
        \caption{Right handed trefoil}
    \end{subfigure}
    \begin{subfigure}[t]{0.2\textwidth}
        \includegraphics[width=0.8\textwidth]{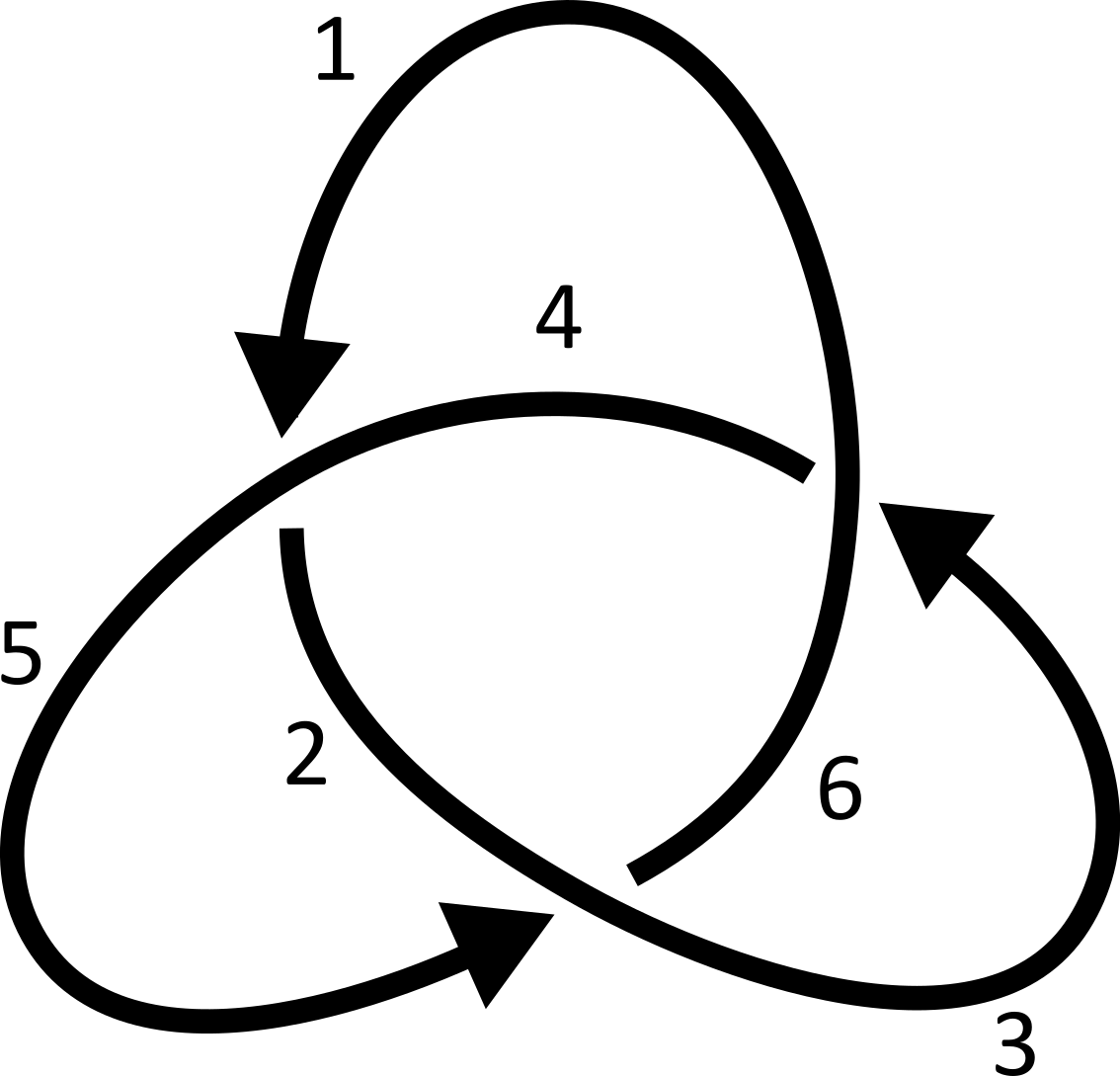}
        \caption{Right handed trefoil with edges labeled}
    \end{subfigure}
    \begin{subfigure}[t]{0.2\textwidth}
        \includegraphics[width=0.8\textwidth]{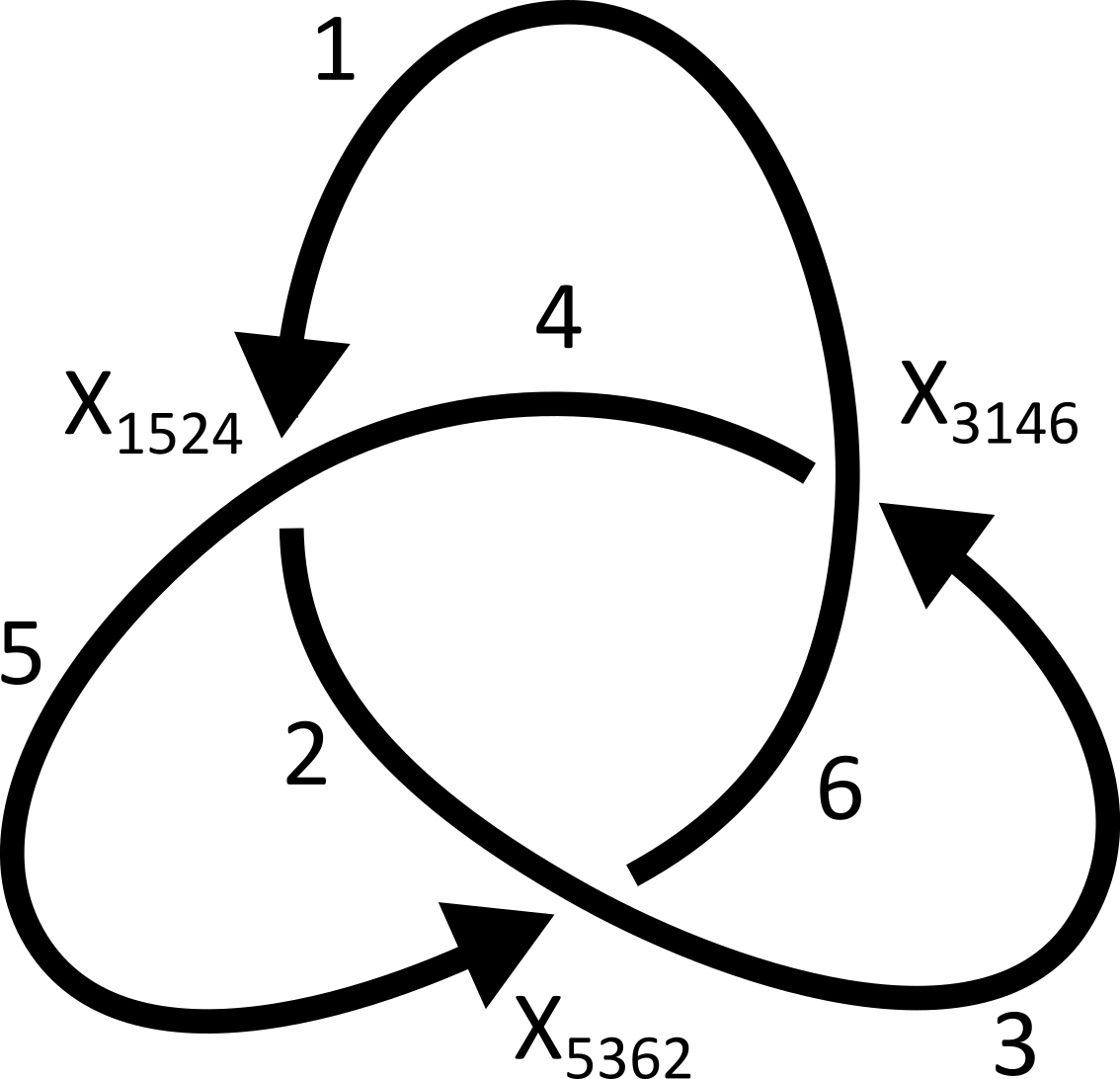}
        \caption{Right handed trefoil with edges and crossings labeled}
    \end{subfigure}
    \caption{Three representations of the right handed trefoil.}
    \label{fig:trefoil}
\end{figure}

Let $L$ be an oriented link projection, $\chi(L)$ an ordered set of crossings of $L$, $n=|\chi(L)|$ the number of crossings, $n_+$ the number of positive crossings, and $n_-$ the number of negative crossings. Let $ \KPC$ denote the $0$-smoothing of the crossing $\KPB$ and let $\KPD$ be the $1-$smoothing so that each vertex $\alpha$ in the cube $\{0,1\}^{\chi(L)}$ has a smoothing $S_\alpha$ of $L$ into a collection of cycles with no crossings, such as $S_{101}$ in Figure \ref{fig:v101}, which has the original edge labels retained. We call the collection of all smoothings of $L$ the \textbf{``cube of smoothings."} The edges of the cube correspond to smoothings that differ by only exchanging one $0$-smoothing for a $1$-smoothing or vice versa, such as $S_{101}$ and $S_{111}$. We can denote each edge by an element of $\{0,1,\star\}^{\chi(L)}$.  Moreover, we can give a direction to the edge by assigning the source to be the smoothing in which $\star$ is replaced by $0$ and the target to be where $\star$ is replaced by $1$. Then $\xi=1\star 1$ is an edge from $S_{101}$ to $S_{111}$.

\begin{figure}[htpb]
    \centering
    \includegraphics[width=0.2\textwidth]{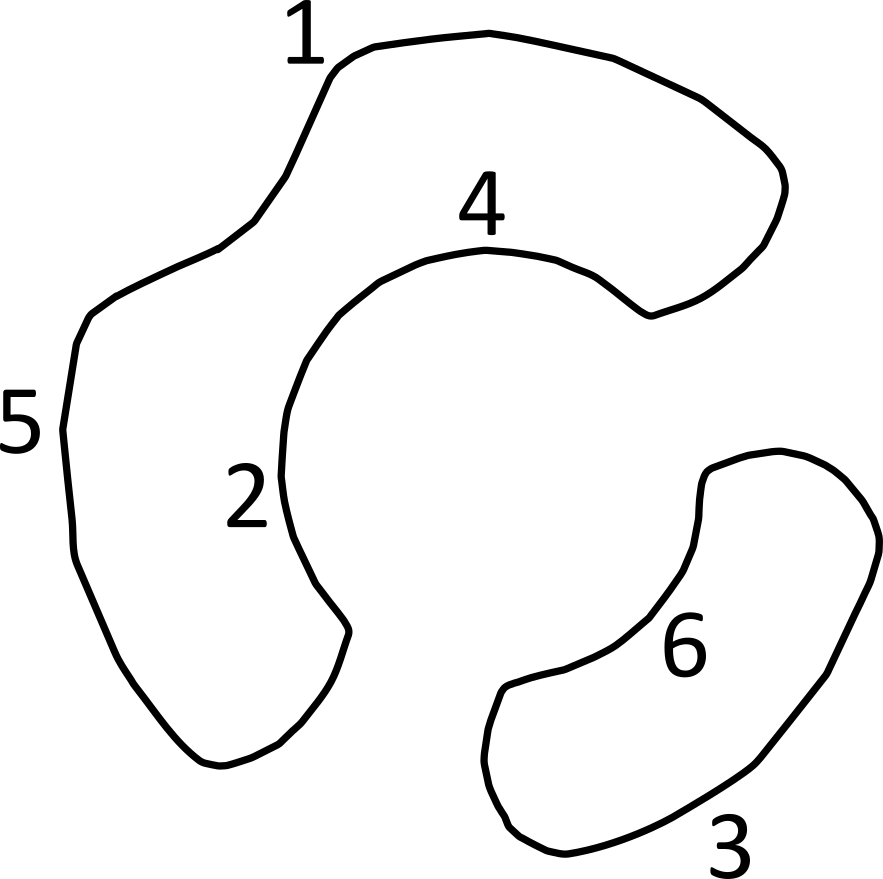}
    \caption{Smoothing $S_{101}$ of the labeled trefoil in Figure \ref{fig:trefoil}(b), with edge labels retained. The left, larger cycle is ``$1$" since the lowest edge label is $1$, and the right, smaller cycles is ``$3$" since the lowest edge label is $3$.}
    \label{fig:v101}
\end{figure}

With the edges and crossings in the non-smoothed diagram labeled, we can systematically label the cycles in a smoothing of the diagram. This will allow us to distinguish basis elements in the chain complex we will create. We use the convention of \cite{Bar_Natan_2002}, which is to label each cycle in a smoothing $S_\alpha$ by the edge in that cycle with the lowest label. In Figure \ref{fig:v101}, the larger cycle on the top left is therefore ``$1$" and the smaller cycle is ``$3$".

From these smoothings, we can create the graded vector spaces that we will stitch together to form the cochain complex. For $W=\bigoplus_j W_j$ a graded vector space with homogeneous components $\{W_j\}$, the \textbf{graded dimension} of $W$ is $\widehat{\dim} W=\sum_j q^j \dim W_j$. Most often in Khovanov homology, the base field is taken to be $\mathbb{Z}_2, \mathbb{Q}$, or generalized to modules over the polynomial ring $\mathbb{Z}[c]$. We will ultimately be using $\mathbb{R}$ for its standard inner product structure. Next, define a \textbf{degree shift} operation by $W\{l\}_j=W_{j-l}$ so that $\widehat{\dim} W\{l\}=q^l \widehat{\dim} W$. Let $V$ be a graded vector space with two basis elements $v_\pm$ of degree $\pm 1$, so that $\widehat{\dim} V=q^{-1}+q$. For $\alpha\in\{0,1\}^{\chi(L)}$, written $\alpha=(\alpha_1,\alpha_2,\dots,\alpha_n)$, define the height $\ell(\alpha)=\sum_{i=1}^n \alpha_i$, which is the number of ones occurring in $\alpha$. Next, let $c(\alpha)$ be the number of cycles in $S_\alpha$, and let $V_\alpha(L)=V^{\otimes c(\alpha)}\{\ell(\alpha)\}$. Since we have a labelling system for each cycle in a smoothing $S_\alpha$ and we have one copy of $V$ for each cycle, we can label the copies of $V$. For the trefoil with smoothing $S_{101}$ in Figure \ref{fig:v101}, we have $V_{101}=V_1\otimes V_3\{2\}$.

For any cochain complex $A$ given by $\cdots\to A^n\xrightarrow[]{d^n} A^{n+1}\to\cdots$, define the \textbf{height shift} of $A$ by $s$, denoted $A[s]$, to be $(A[s])^n = A^{n-s}$, with differentials shifted accordingly. The \textbf{Khovanov bracket} of $L$ is the chain complex $\llbracket L\rrbracket$ with the $r$th chain group defined by 

\[
    \llbracket L \rrbracket^r = \bigoplus_{\ell(\alpha)=r} V_\alpha(L),
\]

and we define a \textbf{normalized chain complex} by

\[
    \mathcal{C}(L) = \llbracket L \rrbracket[-n_{-}]\{n_{+}-2n_{-}\}.
\]

In our example of the right-handed trefoil, there are $3$ positive crossings and $0$ negative crossings, so $\mathcal{C}(L)= \llbracket L\rrbracket[0]\{3\}$. To define the differentials of this cochain complex, we first define a multiplication map, $m:V\otimes V\to V$, and a comultiplication map, $\Delta:V\to V\otimes V$:

\begin{align*}
    m(v_+\otimes v_+) &= v_+\\
    m(v_-\otimes v_-) &= 0\\
    m(v_+\otimes v_-) &= m(v_-\otimes v_+) = v_-\\
    \Delta(v_+) &= v_+\otimes v_- + v_-\otimes v_+\\
    \Delta(v_-) &= v_-\otimes v_-.
\end{align*}

The multiplication can be identified with a cobordism merging the cycles and the comultiplication to splitting cycles when moving along an edge in the cube of smoothings of $L$. We use the standard notation $\Delta$ here for the comultiplication and elsewhere in this paper we use the standard notation $\Delta_n$ for the Khovanov Laplacian, though they are unrelated. We will not need to refer to the comultiplication directly outside of this subsection.

To build the chain complex differentials, we first build a differential $d_\xi:V_\alpha(L)\to V_\beta(L)$ for each edge $\xi$ from $S_{\alpha}$ to $S_{\beta}$ in the cube of smoothings. We apply the multiplication and comultiplication maps to the tensor factors of $V_\alpha$ that correspond to cycles that merge or split. More explicitly, when transitioning from a smoothing $S_\alpha$ to $S_\beta$, either one of the cycles of the smoothing splits and we define $d_\xi$ to be the comultiplication on the corresponding component of $V_\alpha$, or two of the cycles merge and we define the $d_\xi$ to be the multiplication on the corresponding components of $V_\alpha$. On the tensor factors corresponding to cycles that do not split or merge, we define $d_\xi$ to be the identity. 

This differential is a map $d_\xi:V_\alpha\to V_\beta$, and now we want to extend to a differential $d^{r}:\llbracket L \rrbracket^{r}\to\llbracket L \rrbracket^{r+1}$ on the chain complex $\llbracket L \rrbracket$, and ultimately to $\mathcal{C}(L)$. Let $\xi_i$ be the element at position $i$ in $\xi$, and let $j$ be the index of $\star$ in $\xi$.  We will assemble all of the differentials on the chain complex $\llbracket L \rrbracket$ from alternating sums of the differentials $d_{\xi}$, where the sign is given by $\mathrm{sgn}(\xi)=(-1)^{\sum_{i<j}\xi_i}$, which is positive if there are an even count of ones before the $\star$ and negative if the count of ones is odd. We will also set $|\xi| = \ell(\alpha)$, where $\alpha$ is the smoothing where $\star$ is replaced by $0$.  Then we define $d^r:\llbracket L \rrbracket^{r}\to\llbracket L \rrbracket^{r+1}$ by

\[
    \sum_{|\xi|=r} \mathrm{sgn}(\xi)d_\xi.
\]

The fact that $d^{r+1}\circ d^{r}=0$ and the motivation for this system of assigning signs is discussed in more detail in \cite{Bar_Natan_2002}. This differential is the final ingredient we need to call $\llbracket L \rrbracket$ a cochain complex, and therefore also $\mathcal{C}(L)$ is a cochain complex. We may then compute the homology of $\mathcal{C}(L)$ 

\begin{definition}
    The $\boldsymbol{r}$\textbf{-th Khovanov homology} is the $r$-th cohomology of the chain complex $\mathcal{C}(L)$,

\[
    H^r(L) = \ker d^{r} / \Ima d^{r-1},
\]   

and we define the \textbf{graded Poincar\'e polynomial}
\[
    Kh(L) = \sum_{r}t^r \widehat{\dim} H^{r}(L).
\]

The $q$-graded components of $H^r(L)$ are written $H^{r,q}(L)$. When we wish to emphasize the coefficient field $\mathbb{K}$, we write $H_{\mathbb{K}}^r(L)$ and $H_{\mathbb{K}}^{r,q}(L).$ 
\end{definition}

We can now state Khovanov's main results \cite{khovanov_categorification_2000}.

\begin{theorem}
    The graded dimensions of $H^{r}(L)$ are link invariants.
\end{theorem}

\begin{theorem}
    The Poincar\'e polynomial $Kh(L)$ at $t=-1$ is the un-normalized Jones polynomial of $L$.
\end{theorem}

We will need to explicitly compute the differentials for the Laplacian we construct, so here we explain the computation of a differential in the standard case. For practical and efficient computations, cancellation techniques are often used \cite{Bar-Natan_Fast}. We will not do such reductions, because they could impact the non-harmonic spectra of the Laplacian, and we want to first establish the most direct construction before exploring variants. To compute a differential $d^r:\llbracket L \rrbracket^{r}\to \llbracket L \rrbracket^{r+1}$, we first write a basis for each $\llbracket L \rrbracket^r=\bigoplus_{\ell(\alpha)=r}V_\alpha(L)$, which we can assemble from a basis of each $V_\alpha(L)$. For example, we can label the basis elements of $V_3(L)$ by $v_{+}^{3}$ and $v_{-}^{3}$ and similarly for $V_1$, so that a basis of $V_{101}$ is 

\[\{v_{+}^{1}\otimes v_{+}^{3},v_{-}^{1}\otimes v_{+}^{3}, v_{+}^{1}\otimes v_{-}^{3}, v_{-}^{1}\otimes v_{-}^{3} \}.
\]

For convenience we will sometimes omit the tensor product, e.g. $v_{+}^{1}\otimes v_{-}^{3} = v_{+}^{1}v_{-}^{3}$. The full chain group $\llbracket L \rrbracket^r$ is a direct sum of these tensor products, so a basis of $\llbracket L \rrbracket^r$ can be obtained from the union of these bases. To clarify which smoothing a basis element corresponds to, we may label it as $V_{101}v_{+}^{1}v_{-}^{3}$. There are three edges associated with $S_{101}$: $(\star 01), (1\star 1),$ and $(10\star)$, but only $\xi=1\star1$ is outgoing to $S_{111}$, so we compute $d_{1\star 1}:V_{101}\to V_{111}$, or

\[
    d_{1\star 1}:V_1\otimes V_3\{2\}\to V_1\otimes V_2\otimes V_3\{3\}.
\]

Cycle $1$ of $S_{101}$ splits to give two cycles in $S_{111}$, so $d_{1\star 1}$ is the identity on $V_3$ and comultiplication on $V_1$. For example, 

\begin{align*}
d_{1\star 1}(v_{+}^{1}\otimes v_{-}^{3}) &= (v_{+}^{1}\otimes v_{-}^{2} + v_{-}^{1}\otimes v_{+}^{2} )\otimes v_{-}^{3}\\
 &= v_{+}^{1}\otimes v_{-}^{2}\otimes v_{-}^{3} + v_{-}^{1}\otimes v_{+}^{2}\otimes v_{-}^{3}.
\end{align*}

As an example of a boundary map $d^r$ on the full chain complex, the chain groups $\llbracket L \rrbracket^0$ and $\llbracket L \rrbracket^1$ have dimensions $4$ and $6$, respectively, and so $d^0:\llbracket L \rrbracket^0\to\llbracket L \rrbracket^1$ is a $6\times 4$ matrix. To calculate this, we first observe that smoothing $S_{000}$ has two components, and all of the edges originating at it, $0 0\star$, $0\star 0$, and $\star 0 0$, all point toward smoothings with one component. The sign of the differential for each edge will be $1$, since there are zero ones before the $\star$ in each edge. In each case, both cycles merge, so the differential is defined to be the multiplication map $m:V\otimes V\to V$. For this calculation, we subscript $m$ by the edge it corresponds to, e.g. $m_{0\star 0}$, to clarify the connection between steps. For example,

\begin{align*}
    d^0(V_{000}v_{-}^{1}v_{+}^{2}) &= d_{0 0\star}(v_{-}^{1}v_{+}^{2}) + d_{0\star 0}(v_{-}^{1}v_{+}^{2}) + d_{\star 0 0}(v_{-}^{1}v_{+}^{2})\\
    &= m_{0 0\star}(v_{-}^{1}v_{+}^{2}) + m_{0\star 0}(v_{-}^{1}v_{+}^{2}) + m_{\star 0 0}(v_{-}^{1}v_{+}^{2})\\
    &= V_{001}v_{-}^{1} + V_{010}v_{-}^{1} + V_{001}v_{-}^{1},
\end{align*}

where the final equality comes from the definition of multiplication $m(v_-\otimes v_+)=m(v_+\otimes v_-)=v_-$. The exact same calculation works for $d^0(V_{000}v_{+}^{1}v_{-}^{2})$. By using the definition $m(v_+\otimes v_+)=v_+$ and a similar calculation we can obtain $d^0(V_{000}v_{+}^{1}v_{+}^2)$, and since $m(v_-\otimes v_-)=0$ we find that $d^0(V_{000}v_{-}^{1}v_{-}^{2})=0$. This can be represented as a matrix,

\begin{align*}
    d^{0} &= \begin{blockarray}{cccccc}
          &  &V_{000}             & V_{000} & V_{000} & V_{000} \\
          &  & v_{-}^1 v_{-}^{2} &v_{-}^1 v_{+}^{2} & v_{+}^1 v_{-}^{2}& v_{+}^1 v_{+}^{2}\\
            \begin{block}{cc(cccc)}
              V_{001} & v_{-}^{1} & 0 & 1 & 1 & 0 \\
              V_{001} & v_{+}^{1} & 0 & 0 & 0  & 1 \\
              V_{010} & v_{-}^{1} & 0 & 1 & 1  & 0 \\
              V_{010} & v_{+}^{1} & 0 & 0 & 0 & 1 \\
              V_{100} & v_{-}^{1} & 0 & 1 & 1  & 0 \\
              V_{100} & v_{+}^{1} & 0 & 0 & 0  & 1 \\
        \end{block}
        \end{blockarray}.
\end{align*}

The kernel of this map is spanned by $V_{000}v_{-}^{1}v_{-}^{2}$ and $V_{000}v_{-}^{1}v_{+}^{2}-V_{000}v_{+}^{1}v_{-}^{2}$. The height shift in $\mathcal{C}(L)=\llbracket L\rrbracket\{n_{-}\}[n_{+}-2n_{-}]$ is trivial here, so $d^0:\llbracket L \rrbracket^0\to\llbracket L \rrbracket^1$ and $d^0:\mathcal{C}^0(L)\to \mathcal{C}^1(L)$ differ only by the degree shift. Since $\llbracket L\rrbracket^0$ is the chain group with least possible number of $1$-smoothings, $H^0(L)=\ker d^0$, so $\dim H^0(L)=2$. 

\section{Khovanov Laplacian and   Khovanov Dirac}\label{sec:khovanovLaplacian}

\subsection{Construction}
Here we define the Khovanov Laplacian and Khovanov Dirac. To do so we must fix our base field as $\mathbb{R}$. Note that the Khovanov homology is defined with respect to any field, $\mathbb{Z}$, or $\mathbb{Z}[c]$, but most commonly $\mathbb{Z}, \mathbb{F}_2,$ and $\mathbb{Q}.$ To build a Khovanov Laplacian and Khovanov Dirac, the main structure we require of our coefficient system is an inner product space to produce a well-defined adjoint of the boundary map, so we will use coefficients in $\mathbb{R}$. The construction could be recast in terms of rational coefficients and rational Khovanov homology, but it is generally convenient for the eigenvalues to be in the same field as our coefficients. For the remainder of the paper, we assume that each $V$ is a graded vector space over $\mathbb{R}$, and hence the groups $\llbracket L \rrbracket^r$ and $\mathcal{C}^r(L)$ are tensor products and direct sums of vector spaces over $\mathbb{R}$. Similarly, we take homology with coefficients in $\mathbb{R}$.   
The part of the Khovanov homology we want to produce from the kernel of our Laplacian is the graded components $H^{r,q}(L)$, which give the link invariants. We will therefore need to explicitly compute chains at height $r$ and degree $q$, $\mathcal{C}^{r,q}(L)$, which are the $q$-graded subspaces of $\mathcal{C}^{r}(L)$, i.e. $\mathcal{C}^{r}(L)=\bigoplus_{q}\mathcal{C}^{r,q}(L)$. The Khovanov complex differentials are constructed to be of quantum degree $0$, so all degree $q$ elements of $\mathcal{C}^{r}(L)$ will have boundaries that are degree $q$ elements of $\mathcal{C}^{r+1}(L)$. Then we can define our boundary maps via restriction of $d^r$ to  $d^{r,q}:\mathcal{C}^{r,q}(L)\to \mathcal{C}^{r+1,q}(L)$. This leads to the following lemma:

\begin{lemma} The complex
    \[
    \cdots\to \mathcal{C}^{r-1,q}(L)\xrightarrow[]{d^{r-1,q}} \mathcal{C}^{r,q}(L)\xrightarrow[]{d^{r,q}} \mathcal{C}^{r+1,q}(L)\to\cdots
    \]

    is a cochain complex. Moreover, the homology of this cochain complex is the graded Khovanov homology,

    \[
        \ker d^{r,q}/\Ima d^{r-1,q} \cong H^{r,q}(L).
    \]
    \label{lemma:subcomplex}
\end{lemma}
\begin{proof}
    As $d^{r,q}$ is a restriction of $d^r,$ we know that $d^{r+1,q}\circ d^{r,q}=0$, making the complex a cochain complex. The isomorphism follows from the additivity of the homology functor applied to $\mathcal{C}^{r}(L)=\bigoplus_{q}\mathcal{C}^{r,q}(L)$.
\end{proof}

We now construct the adjoint maps, Khovanov Dirac, and Khovanov Laplacian for this complex as done in subsection \ref{subsec:laplacian}, so that the kernel of the Laplacian will produce $H^{r,q}(L)$.

First, endow $\mathcal{C}^{r,q}(L)$ with an inner product by specifying it on basis elements. A basis of $\mathcal{C}^{r,q}(L)$ is the restriction of a basis of $\mathcal{C}^{r}(L)$ to the elements of degree $q$. Suppose they are $B^{r,q}=\{e_1,\dots,e_k\}$. Then define

\[
    \langle\cdot,\cdot\rangle:B^{r,q}\times B^{r,q}\to\mathbb{R}
\]

by 

\[
    \langle e_i,e_j\rangle = \begin{cases}
        1, & \text{if } i=j\\
        0, & \text{else}
    \end{cases},
\]

and extend this inner product to all of $\mathcal{C}^{r,q}(L)$. We also need a suitable notion of the dual complex. Khovanov discusses one notion of the dual of the entire graded complex and cube of resolutions in \cite{khovanov_categorification_2000} to show isomorphisms related to the mirror of a knot. For the purpose of constructing the Laplacian, we will simply dualize each of the $q$-graded subcomplexes $\mathcal{C}^{\bullet,q}(L)$. This gives the dual spaces $\Hom(\mathcal{C}^{r,q}(L),\mathbb{R})$ with basis $\{\phi_1,\dots,\phi_k\}$, where $\phi_i$ is defined by $\phi_i(e_i)=1$ and $\phi_i(e_j)=0$ for $i\neq j$. This agrees with the canonical isomorphism with the dual of an inner product space. We also need an inner product on the dual spaces, which we can define for any $f,g\in \Hom(\mathcal{C}^{r,q}(L),\mathbb{R})$ by 

\[
\llangle f,g\rrangle = \sum_{e\in B^{r,q}} f(e)g(e).
\]

With the inner products on $\mathcal{C}^{r,q}(L)$ and its dual in hand, we now have that there is a unique adjoint of $d^{r,q}$ with respect to these inner products,   $\left(d^{r,q}\right)^*:\mathcal{C}^{r+1,q}(L)\to \mathcal{C}^{r,q}(L)$, for each $r$ and $q$.

Our cochain complex diagram becomes

\begin{center}        
    \begin{tikzcd}
    \cdots
    \arrow[r, "", shift left]
    &
    \mathcal{C}^{r-1,q}(L) \arrow[r, "d^{r-1,q}", shift left] \arrow[l, "", shift left]& \mathcal{C}^{r,q}(L) \arrow[r, "d^{r,q}", shift left] \arrow[l, "\left(d^{r-1,q}\right)^*", shift left] & \mathcal{C}^{r+1,q}(L) \arrow[l, "\left(d^{r,q}\right)^*", shift left] \arrow[r, "", shift left] & \arrow[l, "", shift left] \cdots.
    \end{tikzcd}
\end{center}

We are now ready to introduce the central objects of this work.

\begin{definition} The $\boldsymbol{(r,q)}$\textbf{-Khovanov Dirac} of the link projection $L$ is $\mathcal{D}_{L}^{r,q}:\bigoplus_{i=n_{-}}^{r+1} \mathcal{C}^{i,q}(L)\to \bigoplus_{i=n_{-}}^{r+1} \mathcal{C}^{i,q}(L)$, given by

\begin{equation}\label{Dirac}
    \mathcal{D}^{r,q} = \begin{bmatrix}\mathbf{0}_{m_{ \left(-n_{-}\right) }\times m_{ \left(-n_{-}\right)} } & \left(d^{-n_-,q}\right)^* & \mathbf{0}_{m_{ \left(-n_{-}\right) }\times m_{ \left(-n_{-}+2\right)} } & \cdots & \mathbf{0}_{m_{ \left(-n_{-}\right) }\times m_{r} } & \mathbf{0}_{m_{ \left(-n_{-}\right) }\times m_{r+1} }\\
    											 d^{-n_-,q} & \mathbf{0}_{m_{ \left(-n_{-}+1\right) }\times m_{ \left(-n_{-}+1\right)} } & \left(d^{-n_- + 1,q}\right)^* & \cdots &\mathbf{0}_{m_{ \left(-n_{-}+1\right) }\times m_{r} } &\mathbf{0}_{m_{ \left(-n_{-}+1\right) }\times m_{r+1} }\\
    											 \mathbf{0}_{m_{ \left(-n_{-}+2\right) }\times m_{ \left(-n_{-}\right)} }& d^{-n_-+1,q} & \mathbf{0}_{m_{ \left(-n_{-}+2\right) }\times m_{ \left(-n_{-}+2\right)} } & \cdots & \mathbf{0}_{m_{ \left(-n_{-}+2\right) }\times m_{r} } & \mathbf{0}_{m_{ \left(-n_{-}+2\right) }\times m_{r+1} }\\
        \vdots & \vdots & \vdots & \ddots &\vdots & \vdots\\
        
        \mathbf{0}_{m_{r}\times m_{ \left(-n_{-}\right)}} & \mathbf{0}_{m_{r}\times m_{ \left(-n_{-}+1\right)}} & \mathbf{0}_{m_{r}\times m_{ \left(-n_{-}+2\right)}} &\cdots & \mathbf{0}_{m_{r}\times m_{r} } & \left(d^{r,q}\right)^*\\
         \mathbf{0}_{m_{r+1}\times m_{ \left(-n_{-}\right)}} & \mathbf{0}_{m_{r+1}\times m_{ \left(-n_{-}+1\right)}} & \mathbf{0}_{m_{r+1}\times m_{ \left(-n_{-}+2\right)}} &\cdots & d^{r,q} & \mathbf{0}_{m_{r+1}\times m_{r+1}}\\
    \end{bmatrix},
\end{equation}

	where $m_i=\dim C^{i,q}(L)$.
\end{definition}

The construction is almost algebraically identical to the combinatorial Dirac operator, except here we use the adjoint of the boundary operator where we previously used the boundary operator, due to the fact we are working with a cochain complex, and due to the height shift we begin indexing at $-n_-$, where the least nonempty chain complex resides. This configuration will produce the desired relationship between the Khovanov Dirac and Khovanov Laplacian operators.

\begin{definition} The $\boldsymbol{(r,q)}$\textbf{-Khovanov Laplacian} of the link projection $L$ is

\begin{equation}\label{Laplacian}
    \Delta^{r,q}_L = \left(d^{r,q}\right)^*\circ d^{r,q} + d^{r-1,q}\circ\left(d^{r-1,q}\right)^*.
\end{equation}

We sometimes separately consider $\Delta_{L,\text{up}}^{r,q}= \left(d^{r,q}\right)^*\circ d^{r,q}$ and $\Delta_{L,\text{down}}^{r,q}= d^{r-1,q}\circ\left(d^{r-1,q}\right)^*$, so that $\Delta^{r,q}_L=\Delta_{L,\text{up}}^{r,q}+\Delta_{L,\text{down}}^{r,q}.$

\end{definition}

We can obtain the desired key property of  Khovanov  Laplacians in Theorem \ref{theorem:kernel}.

\begin{theorem} The kernel of the   Khovanov Laplacian is isomorphic to the graded, real Khovanov homology group, i.e.
    \[
    \ker\Delta^{r,q}_L\cong H_{\mathbb{R}}^{r,q}(L).
    \]
    \label{theorem:kernel}
\end{theorem}
\begin{proof}
    By lemma \ref{lemma:subcomplex}, $H^{r,q}(L)$ is the $r$-th homology of the chain complex $\mathcal{C}^{\bullet,q}(L)$. By the same argument for the standard combinatorial Laplacian in \cite{LimHodge2020}, $\ker\Delta^{r,q}_L$ is also isomorphic to the homology of the same chain complex.  
\end{proof}

As immediate consequences, we can re-frame key results of Khovanov homology in terms of the Khovanov Laplacian. 

\begin{corollary} The Poincar\'e polynomial may be expressed in terms of the Laplacian:
\[
Kh_\mathbb{R}(L)(q,t) = \sum_{r,j}t^r q^j \dim \ker \Delta^{r,j}_L.
\]
\end{corollary}

\begin{corollary} The normalized Jones polynomial may be expressed in terms of the Laplacian:
\[
J(L) = \frac{1}{q+q^{-1}}\sum_{r,j}(-1)^r q^j \dim\ker\Delta^{r,j}_L.
\]
\end{corollary}

\begin{definition} We denote the \textbf{spectrum} of $\Delta_L^{r,q}$ in nondecreasing order (as the eigenvalues are again nonnegative) by $S_L^{r,q}=\{\lambda^{r,q}_0,\lambda^{r,q}_1,\dots,\lambda^{r,q}_M\}$ and denote the least nonzero eigenvalue by $\lambda^{r,q}$ or just $\lambda$. 
\end{definition}

Now for $K$ the link diagram of the right-handed trefoil in Figure \ref{fig:trefoil}, we compute the Laplacians $\Delta^{0,3}_K$ and $\Delta^{1,5}_K$ and their spectra. The full collection of spectra for $K$ are in Table 1 of the supplementary information. The least nonzero eigenvalues are plotted as a heatmap in Figure \ref{fig:trefoil_heatmap}, for two different trefoil projections, the ``standard" one being the one used here. The ranks of the Khovanov homology are also given in Figure \ref{fig:trefoil_heatmap} for comparison.

First, we calculate $\Delta^{0,3}_K$, for which we would need $d^{-1,3}$ and $d^{0,3}$. The chain group $\mathcal{C}^{-1,3}(K)$ is zero, so $d^{-1,3}=0$ and $\Delta^{0,3} = \left(d^{0,3}\right)^*\circ d^{0,3}$. To clarify the matrix form of the differentials, we label the rows and columns by the basis elements of $\mathcal{C}^{r,q}(L)$, with both the labeled vector space $V_\alpha$ and the corresponding generators of each $V$.

\begin{align*}
    d^{0,3} &= \begin{blockarray}{cccc}
          &  &V_{000} & V_{000}  \\
          &  & v_{+}^{1}v_{-}^{2} & v_{-}^{1}v_{+}^2   \\
            \begin{block}{cc(cc)}
              V_{001} &v_{-}^{1} & 1 & 1  \\
              V_{010} &v_{-}^{1} & 1 & 1 \\
              V_{100} &v_{-}^{1} & 1 & 1 \\
        \end{block}
        \end{blockarray}.
\end{align*}

Then, according to Eq. (\ref{Laplacian}), the Laplacian matrix is  

\begin{align*}
    \Delta^{0,3}_K &= \begin{pmatrix}
        1 & 1  & 1\\
        1 & 1 & 1
    \end{pmatrix}\begin{pmatrix}
        1 & 1\\
        1 & 1\\
        1 & 1
    \end{pmatrix}\\
    &= \begin{pmatrix}
        3 & 3\\
        3 & 3
    \end{pmatrix}.
\end{align*}

This matrix has spectrum $S^{0,3}_K = \{0,6\}$, which also gives $\beta^{0,3}_K = \dim\ker \Delta^{0,3}_K = \dim H^{r,q}_\mathbb{R}(K) = 1$ and the least nonzero eigenvalue is $\lambda=6.$

To do the same for $\Delta^{1,5}_K$, we need $d^{0,5}$ and $d^{1,5}$. They are

\begin{align*}
        d^{0,5} &= \begin{blockarray}{ccc}
          &   & V_{000}  \\
          &  & v_{+}^1V_{+}^2 \\
            \begin{block}{cc(c)}
                V_{001}  &  v_{+}^{1} & 1\\
                V_{010}  &  v_{+}^{1} & 1\\
                V_{100}  &  v_{+}^{1} & 1\\
        \end{block}
        \end{blockarray}\\
    d^{1,5} &= \begin{blockarray}{ccccc}
          &  &V_{001}             & V_{010} & V_{100}\\
          &  & v_{+}^{1} &v_{+}^{1} & v_{+}^{1} \\
            \begin{block}{cc(ccc)}
                V_{011} & v_{+}^{1}v_{-}^3 & 1 & -1 & 0\\
                V_{011} & v_{-}^{1}v_{+}^3 & 1 & -1 & 0\\
                V_{101} & v_{+}^{1}v_{-}^2 & 1 & 0 & -1\\
                V_{101} & v_{-}^{1}v_{+}^2 & 1 & 0 & -1\\
                V_{110} & v_{+}^{1}v_{-}^2 & 0 & 1 & -1\\
                V_{110} & v_{-}^{1}v_{+}^2 & 0 & 1 & -1\\
        \end{block}
        \end{blockarray}.
\end{align*}

Then, according to Eq. (\ref{Laplacian}), the Laplacian matrix is  

\begin{align*}
    \Delta^{1,5}_K &= \begin{pmatrix}
      1 \\
      1 \\
      1
    \end{pmatrix}\begin{pmatrix}
       1 & 1 & 1
    \end{pmatrix} + \begin{pmatrix}
        1 & 1   & 1  & 1  & 0 & 0\\
        -1 & -1 & 0 & 0   & 1 & 1\\
        0 & 0   & -1 & -1 & -1 & -1
    \end{pmatrix}\begin{pmatrix}
        1 & -1 & 0\\
        1 & -1 & 0\\
        1 & 0 & -1\\
        1 & 0 & -1\\
        0 & 1 & -1\\
        0 & 1 & -1
    \end{pmatrix}\\
    &= \begin{pmatrix}
        1 & 1 & 1\\
        1 & 1 & 1\\
        1 & 1 & 1
    \end{pmatrix} + \begin{pmatrix}
        4 & -2 & -2\\
        -2 & 4 & -2\\
        -2 & -2 & 4
    \end{pmatrix}\\
    &= \begin{pmatrix}
        5 & -1 & -1\\
        -1 & 5 & -1\\
        -1 & -1 & 5
    \end{pmatrix}.
\end{align*}

The eigenvalues of this matrix are $S_K^{1,5}=\{3,6,6\}$. All of the nonempty spectra for $K$ are reported in the supporting information Table 1. The least nonzero eigenvalues for each grading are reported in Figure \ref{fig:trefoil_heatmap}.

\begin{figure}[htbp]
    \centering
    \begin{subfigure}{0.2\textwidth}
        \centering
        \includegraphics[width=\textwidth]{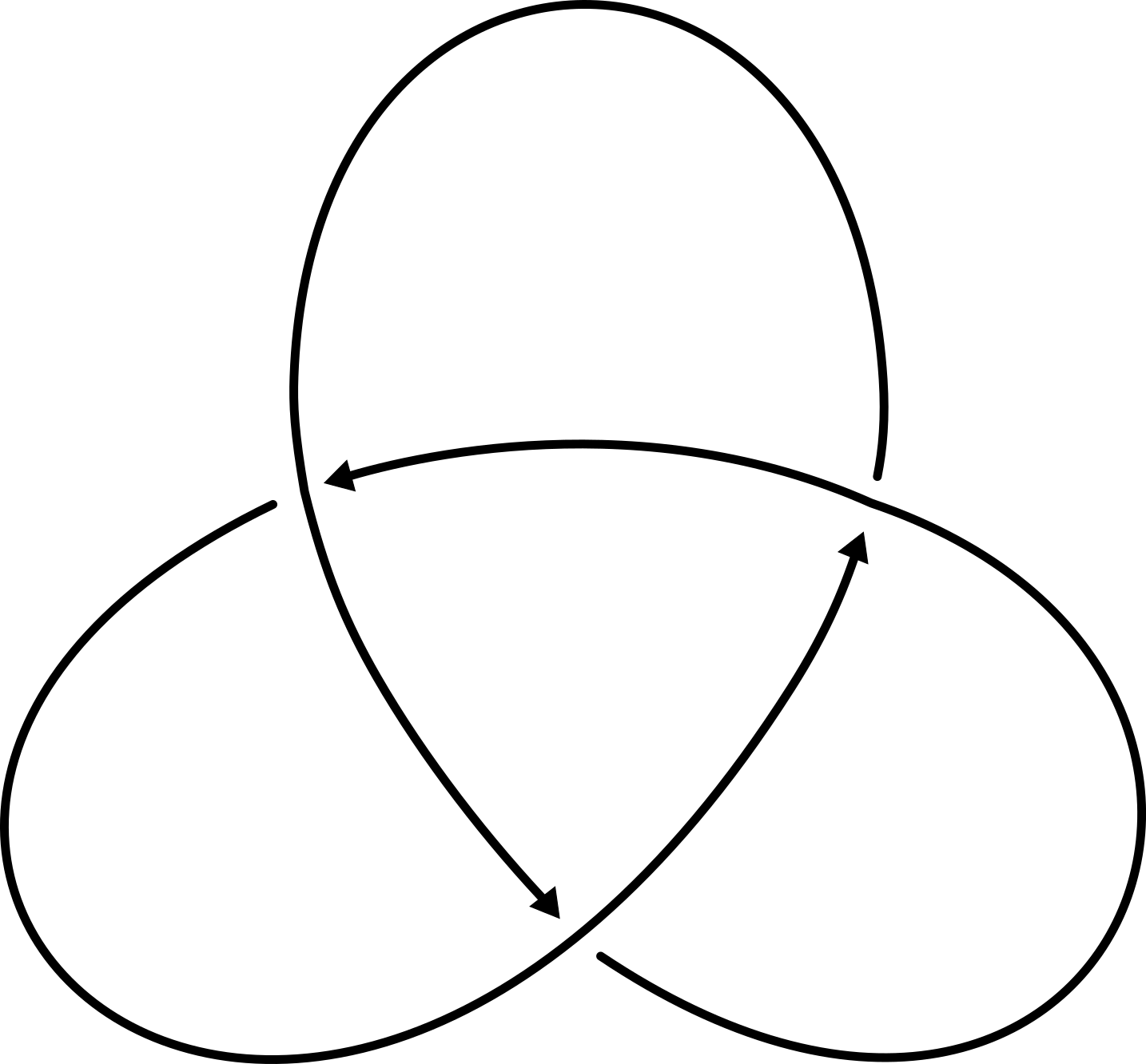}
    \end{subfigure}
    \begin{subfigure}{0.25\textwidth}
        \includegraphics[width=\textwidth]{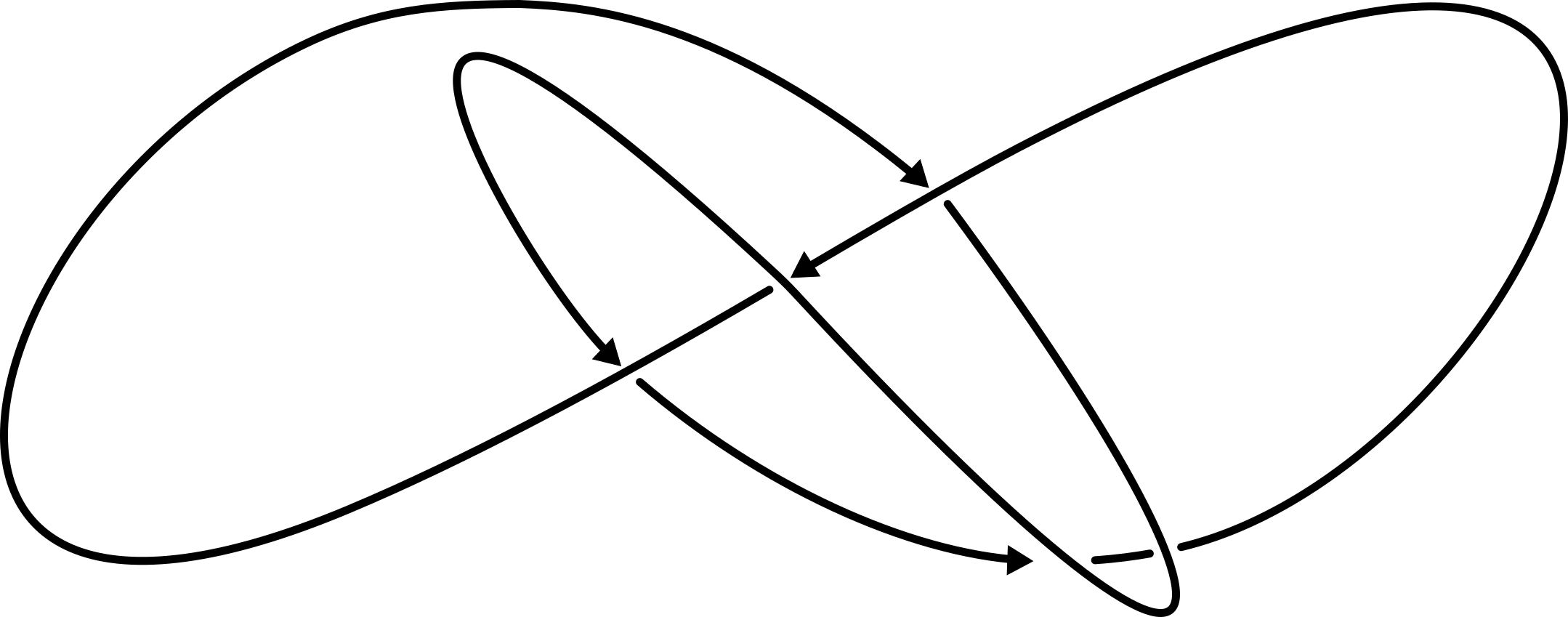}
    \end{subfigure}
    \begin{subfigure}{0.25\textwidth}
        \includegraphics[width=\textwidth]{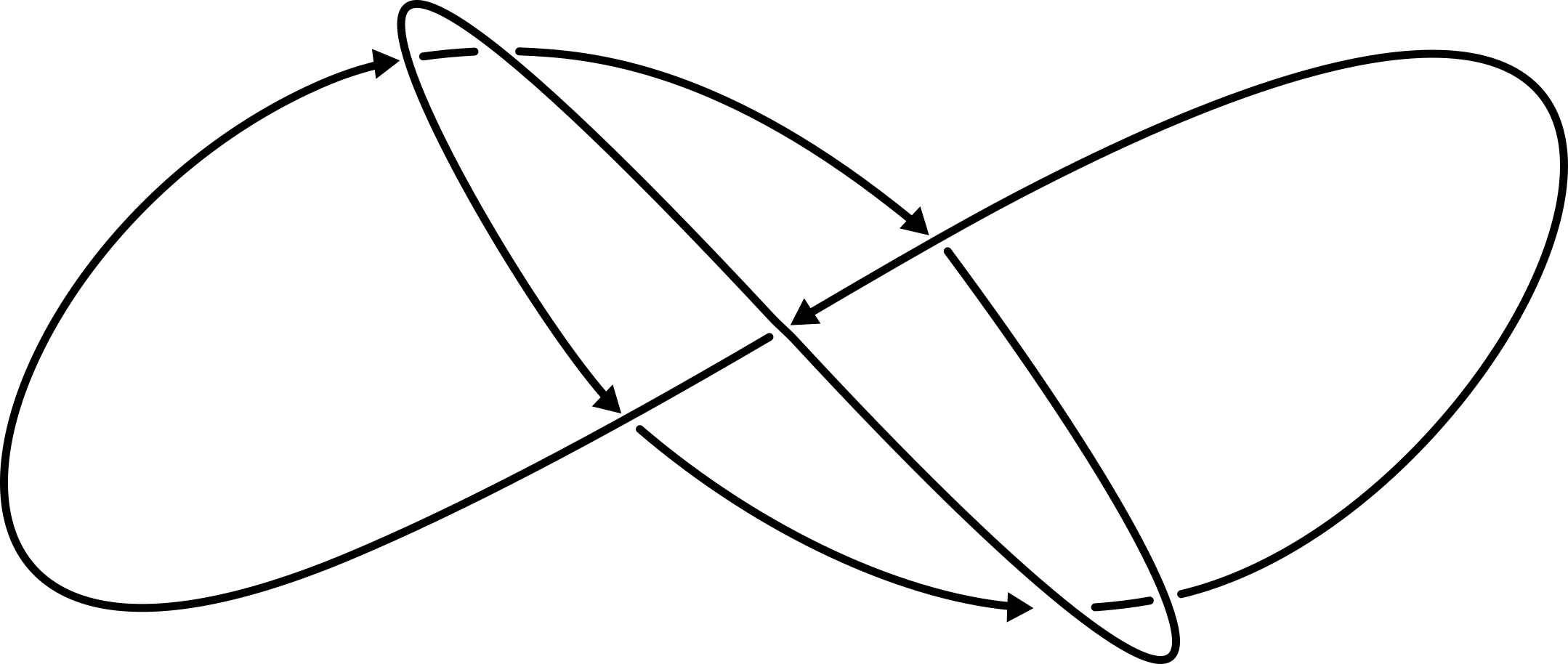}
    \end{subfigure}
    \caption{From left to right, trefoil diagrams showing with the same crossings as the $x-z$ projection of the parametrization in equation \ref{eqn:trefoil}, the $x-y$ projection after one Reidemeister $1$ move is applied, and the $x-y$ projection. These are ordered by the number of crossings.}
    \label{fig:trefoil_right_projections}
\end{figure}

\begin{figure}[htbp]
    \centering
    \includegraphics[width=8cm]{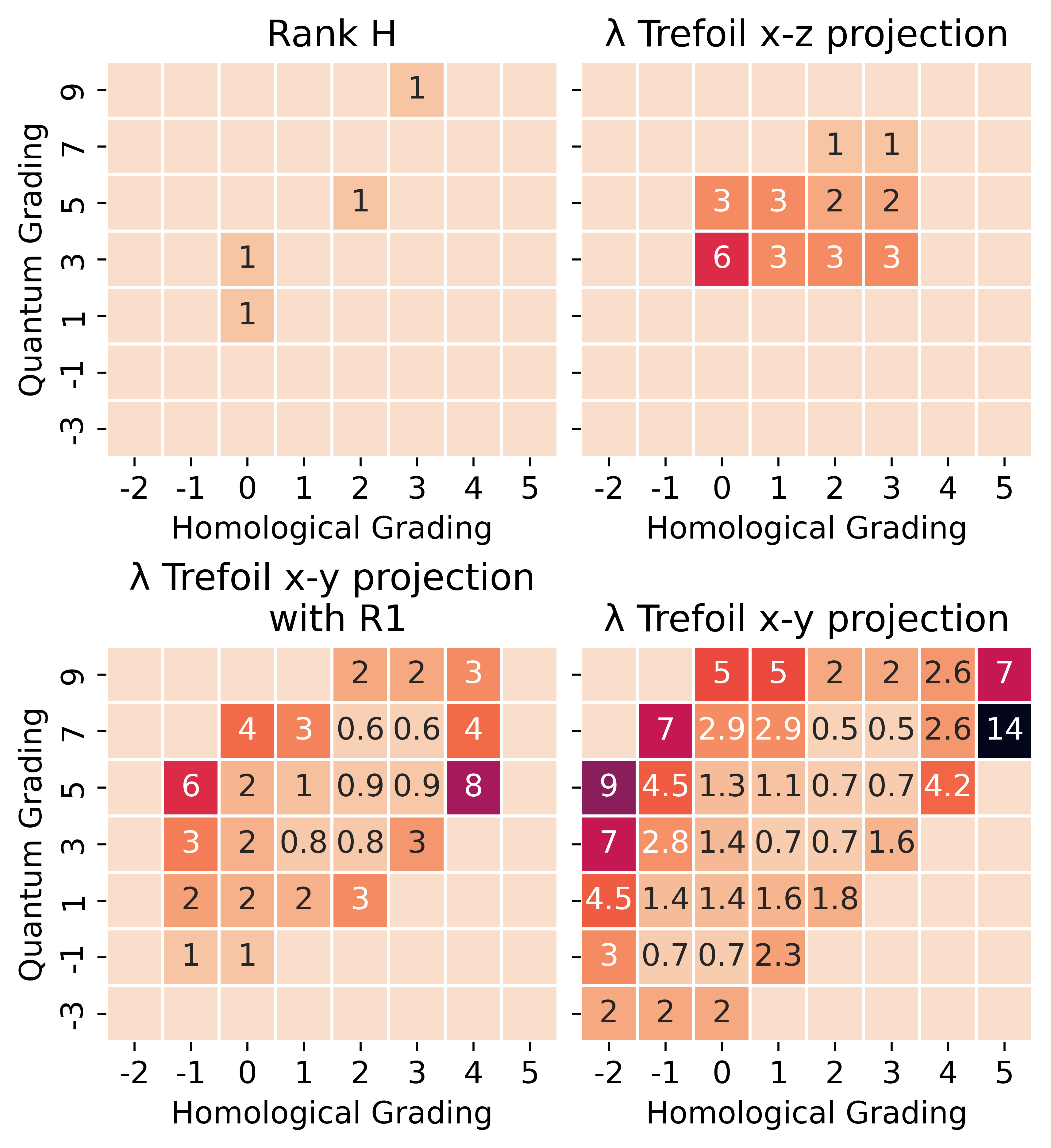}
    \caption{For the right-handed trefoil we show the ranks of the Khovanov homology groups (Top left), the least nonzero eigenvalue of the Khovanov Laplacian of the trefoil using the $x-z$ projection of the parametrization in Eq. (\ref{eqn:trefoil}) (Top right), the $x-y$ projection after one Reidemeister $1$ move is applied (Bottom left), and the $x-y$ projection (Bottom right). These are ordered by the number of crossings.}
    \label{fig:trefoil_heatmap}
\end{figure}

Now we construct an example of a Khovanov Dirac for the same trefoil diagram, 

\[
	\mathcal{D}_K^{1,3} = \begin{bmatrix} \mathbf{0}_{m_0\times m_0} & \left(d_K^{0,3}\right)^* & \mathbf{0}_{m_0\times m_2}\\
										d_K^{0,3} & \mathbf{0}_{m_1\times m_1} & \left(d_K^{1,3}\right)^*\\
										\mathbf{0}_{m_2\times m_0} & d_K^{1,3} & \mathbf{0}_{m_2\times m_2} 
	\end{bmatrix}.
\]

	We previously calculated 
	\[
		d^{0,3}_K = \begin{pmatrix}
        1 & 1\\
        1 & 1\\
        1 & 1
    \end{pmatrix},
    \]
    and we can similarly calculate that 
    
    \[
    		d^{1,3}_K = \begin{pmatrix}
    				1 & -1 & 0\\
    				1 & 0 & -1\\
    				0 & 1 & -1
    		\end{pmatrix}.
    \]
    
    Finally, note that $m_0=\dim C_K^{0,3}=2, m_1=\dim C_K^{1,3}=3$, and $m_2=\dim C_K^{2,3}=3$. Then the Khovanov Dirac is
    
    \[
    		\mathcal{D}_{K}^{1,3} = \begin{bmatrix}
			0 & 0 & 1 & 1 & 1 & 0 & 0 & 0 \\
			0 & 0 & 1 & 1 & 1 & 0 & 0 & 0 \\
			1 & 1 & 0 & 0 & 0 & 1 & 1 & 0 \\
			1 & 1 & 0 & 0 & 0 & -1 & 0 & 1 \\
			1 & 1 & 0 & 0 & 0 & 0 & -1 & -1 \\
			0 & 0 & 1 & -1 & 0 & 0 & 0 & 0 \\
			0 & 0 & 1 & 0 & -1 & 0 & 0 & 0 \\
			0 & 0 & 0 & 1 & -1 & 0 & 0 & 0 
		\end{bmatrix}.
    \]
    
    In the same way as we had for the combinatorial Dirac and Laplacian operators, $\left(\mathcal{D}^{r,q}_L\right)$ is a block diagonal matrix formed by $\Delta^{-n_{-},q}_L,\Delta^{-n_{-}+1,q}_L,\dots,\Delta^{r,q}_{L},\Delta^{r+1,q}_{L,\text{up}}$. In this example we can compute that
    
    \[
    		\left(\mathcal{D}_{K}^{1,3}\right)^2 = \begin{bmatrix}
				3 & 3 & 0 & 0 & 0 & 0 & 0 & 0 \\
				3 & 3 & 0 & 0 & 0 & 0 & 0 & 0 \\
				0 & 0 & 4 & 1 & 1 & 0 & 0 & 0 \\
				0 & 0 & 1 & 4 & 1 & 0 & 0 & 0 \\
				0 & 0 & 1 & 1 & 4 & 0 & 0 & 0 \\
				0 & 0 & 0 & 0 & 0 & 2 & 1 & -1 \\
				0 & 0 & 0 & 0 & 0 & 1 & 2 & 1 \\
				0 & 0 & 0 & 0 & 0 & -1 & 1 & 2 
		\end{bmatrix}.
    \]
    
    The first entry on the block diagonal is the previously-computed $\Delta_K^{0,3},$ and the others are $\Delta_K^{1,3}$ and $\Delta_{K,\text{up}}^{2,3}$. The eigenvalues of $\Delta_{K}^{0,3}$ are $\{0,6\}$, of $\Delta_{K}^{1,3}$ are $\{3,3,6\}$, and of $\Delta_{K,\text{up}}^{2,3}$ are $\{0,3,3\}$. The eigenvalues of $\mathcal{D}_K^{1,3}$ are $\{-\sqrt{6},-\sqrt{3},-\sqrt{3},0,0,\sqrt{3},\sqrt{3},\sqrt{6}\},$ which are in correspondence with the square roots of the eigenvalues of $\Delta_K^{0,3}$, $\Delta_K^{1,3}$, and $\Delta_{K,\textrm{up}}^{2,3}$.

Consider now the trefoil parametrized by the curve in $\mathbb{R}^3$ given by 

\begin{equation} \label{eqn:trefoil}   
(x,y,z) = (\sin u + 2\sin (2u), \sin(3u), 2\cos 2u - \cos(u))
\end{equation}

for $u\in[0,2\pi]$. The three knot diagrams in Figure \ref{fig:trefoil_right_projections} can be obtained from this parametrization in the following way: the standard trefoil diagram by projecting to the $x-z$ plane and at each crossing taking the portion with greater $y$-value to be the over crossing. A distinct diagram, now with $7$ crossing can be found from projecting to the ($x-y$) plane and taking the portion with greater $z$-value to be the over crossing. We obtain the final diagram, now with $5$ crossings, by performing a Reidemeister $1$ move on the $7$-crossing diagram. We compute the Khovanov Laplacian of these diagrams and display the least nonzero eigenvalue in Figure \ref{fig:trefoil_heatmap}, alongside the rank of the Khovanov homology groups. This demonstrates how the same knot can produce a notably distinct set of Khovanov Laplacian eigenvalues, depending on how we represent the knot as a planar diagram. The purpose of this work is not to produce a link invariant, but rather a computational tool to study link diagrams. 

\subsection{Analysis}

We would now like to see if the Khovanov Laplacian provides additional insight over Khovanov homology alone. Khovanov homology may not always detect achirality of a knot, the property that a knot is equivalent to its mirror, which is often also called amphichirality. Formally, the \textbf{mirror} of a knot or link $L$ is a link $\bar{L}$ such that there is an orientation-reversing homeomorphism of the ambient space the link is embedded in. All such links are equivalent. The operation of mirroring a link impacts the corresponding link diagram by switching the roles of over crossings and under crossings. Other knot invariants may distinguish a knot from its mirror better than Khovanov homology can \cite{Watson2007_identical}. We consider the same phenomenon from the perspective of the spectra of the Laplacian. Primarily, we find that the property of the spectrum being ``symmetric" is correlated with achirality in knots $K$ with $10$ or fewer crossings such that the Khovanov homology or Khovanov invariant does not distinguish the knot from its mirror: $Kh(K)=Kh(\bar{K})$. This is expected of an achiral knot, but we would like to distinguish non-achiral (chiral) knots $K$ from their mirror $\bar{K}$.

Concretely, We know $H^{r,q}(K)= H^{-r,-q}(\bar{K})$ for any knot $K$ and if $K$ is an achiral knot this implies $H^{r,q}(K)=H^{-r,-q}(K)$ \cite{khovanov_categorification_2000}. The case where Khovanov homology cannot distinguish a chiral knot from its mirror is when $H^{r,q}(K) = H^{-r,-q}(K)$, but $K\not\simeq \bar{K}$.

\begin{figure}[htbp]
    \centering
    \includegraphics[width=3cm]{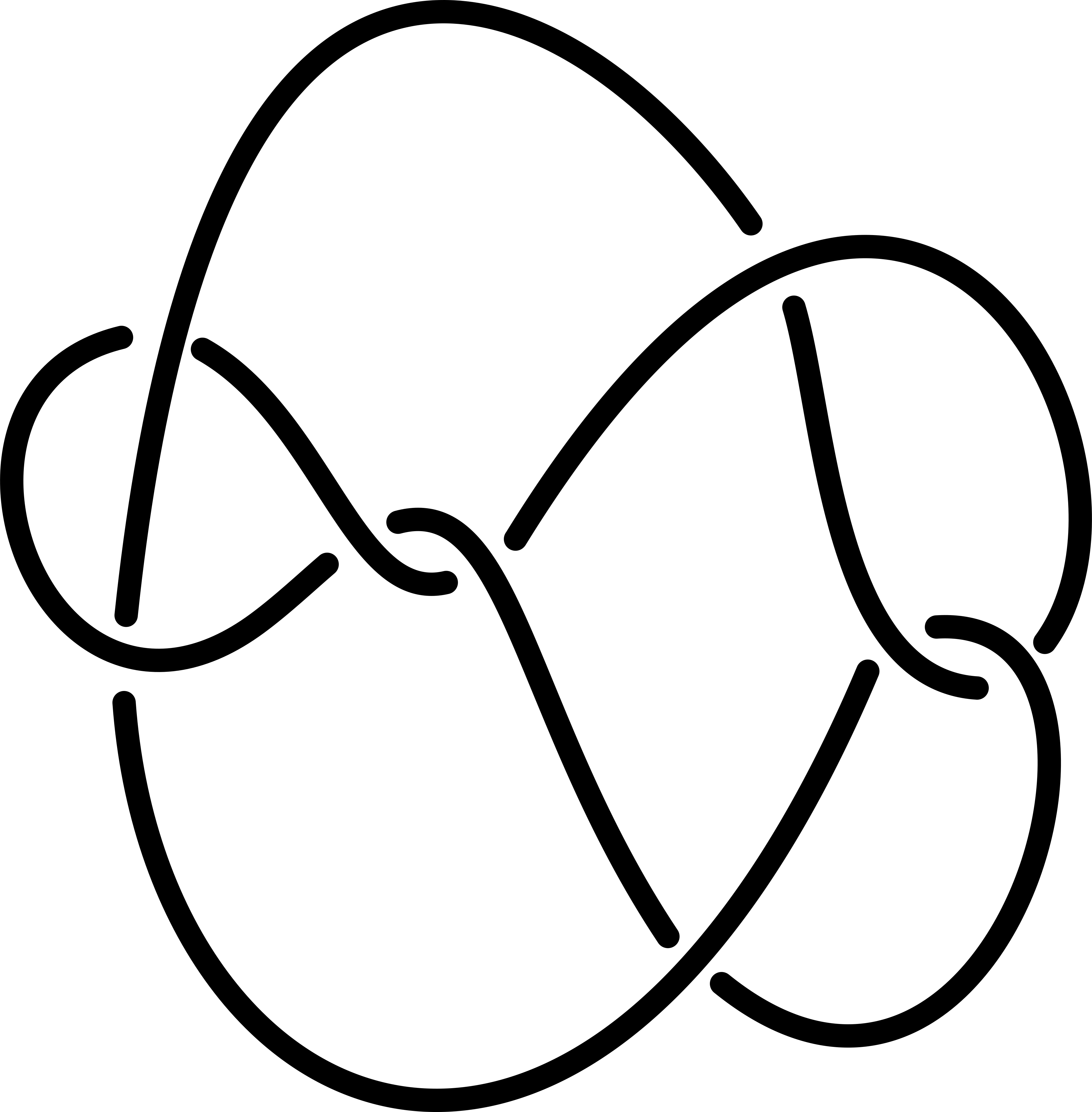}
    \caption{A planar diagram for the prime knot $8_{12}$.}
    \label{fig:8_12}
\end{figure}

Analogously, by the construction of the Laplacian we expect that $S^{r,q}_K=S^{-r,-q}_{\bar{K}}$ as multisets for any $K$, and if $K$ is an achiral knot we might expect $S^{r,q}_K = S^{-r,-q}_K$. This is true for $17$ of the $20$ achiral knots of up to $10$ crossings, with the exceptions being $8_{12}, 10_{43},$ and $10_{37}$. Let us take a closer look at $8_{12}$, with the planar diagram in figure \ref{fig:8_12}, along with its mirror. Because of the result on Khovanov homology, we know that only the non-harmonic spectra may differ. For example, $\dim H^{-4,-9} = \dim H^{4,9} = 1$. However, $S^{-4,-7}_{8_{12}}\neq S^{4,7}_{8_{12}}$. There are several other pairs $(r,q)$ for which $S^{r,q}_{8_{12}}\neq S^{-r,-q}_{8_{12}}$: $(4,5)$, $(4,3)$, $(4,1)$, $(3,7)$ $(3,5),$ $(3,3)$, $(3,1)$, $(2,5)$, $(2,3)$, $(2,1)$, $(2,-1)$, $(1,5)$, $(1,3)$, $(1,1)$, $(1,-1)$, $(1,-3)$, $(0,3)$, and $(0,1)$. They agree only for $S^{4,9}=S^{-4,-9}=\{0\}$, $S^{4,-1}=S^{-4,1}=\{8\}$, $S^{3,-1}=S^{-3,1}=\{8,8,8,8,8,8,10,10\}$, $S^{-2,3}=S^{2,-3}=\{6,6\}$, and $S^{0,5}=S^{0,-5}=\{4,4\}$. We now look closely at $\Delta^{3,7}_{8_{12}}$ and $\Delta^{-3,-7}_{8_{12}}$, a case where these pairs differ when Khovanov homology is nonzero:

\[
	\Delta_{8_{12}}^{3,7} = \begin{bmatrix}
		3  & -1 & 1  & -1 & 1  & -1 & 0 & -1\\
		-1 & 3  & -1 & 1  & -1 & 1  & 0 & 1\\
		1  & -1 & 2  & -1 & 0  & -1 & 0 & 0\\
		-1 & 1  & -1 & 2  & -1 & 0  & 0 & 1\\
		1  & -1 & 0  & -1 & 3  & 0  & 1 & -1\\
		-1 & 1  & -1 & 0  & 0  & 2  & -1 & 0\\
		0  & 0  & 0  & 0  & 1  & -1 & 2  & -1\\
		-1 & 1  & 0  & 1  & -1 & 0  & -1 & 3
	\end{bmatrix}, \Delta_{8_{12}}^{-3,-7} = \begin{bmatrix}
		2 & 1 & 1 & 1 & 0 & 0 & 1 & 0\\
		1 & 3 & 1 & 1 & 0 & 0 & 1 & 1\\
		1 & 1 & 3 & 1 & 0 & 0 & 1 & 1\\
		1 & 1 & 1 & 2 & 0 & 0 & 0 & 1\\
		0 & 0 & 0 & 0 & 3 & 1 & 1 & 1\\
		0 & 0 & 0 & 0 & 1 & 3 & 1 & 1\\
		1 & 1 & 1 & 0 & 1 & 1 & 2 & 0\\
		0 & 1 & 1 & 1 & 1 & 1 & 0 & 2
	\end{bmatrix}.
\]

The spectra for these are 

\begin{align*}
	S^{3,7}_{8_{12}} &= \{0, 0.381966, 1.52265, 2, 2, 2.61803, 4.55193, 6.92542\}\\
	S^{-3,-7}_{8_{12}} &= \{0, 0.41356, 1.48486, 2, 2, 2.76511, 3.70347, 7.63299\}.
\end{align*}

Both of these have $0$ as an eigenvalue with multiplicity $1$, so $\dim H^{3,7}_{8_{12}}=\dim H^{-3,-7}_{8_{12}}=1$. This symmetry, or lack thereof, in the nonzero eigenvalues is independent of the presence of nontrivial Khovanov homology.

We now move from achiral to chiral knots. There are $4$ chiral knots with $10$ crossings such that $H^{r,q}(K)=H^{-r,-q}(K)$ for every $r,q$, so Khovanov Homology does not distinguish the knot from its mirror. They are $10_{48}, 10_{71}, 10_{91},$ and $10_{104}$. We find that for each of these, there is at least one set of non-symmetric spectra, i.e. there exists $r$ and $q$ such that $S^{r,q}(K)\neq S^{-r,-q}(K)$ as multisets. For example, $S^{0,7}_{10_{48}}=\{4, 4.38197, 5.13919, 6.61803, 6.7459, 9.11491\}$, while $S^{0,-7}_{10_{48}}=\{6,6,6,6,6,6\}.$ Also note that since Khovanov homology detects the other chiral knots of $10$ crossings, so too does the Khovanov Laplacian,  via the multiplicity of $0$ as an eigenvalue.

In summary, all of the chiral prime knots up to $10$ crossings that are not distinguished from their mirrors by Khovanov homology give rise to distinct Khovanov Laplacian non-harmonic eigenvalues from their mirrors. Additionally, for  achiral knots,  Khovanov Laplacian has distinct eigenvalues for certain cases. 
Since the spectrum is not a link invariant, we do not claim that a symmetric spectrum implies achirality, but that these properties are correlated and the Khovanov  Laplacian encodes information that is distinct from the Khovanov homology.  

\section{Conclusion}\label{sec:conclusion}

Topological Laplacians, such as combinatorial Laplacians defined on simplicial complexes and Hodge Laplacians defined differentiable manifolds,   inform us about homological group   properties of the spaces via their harmonic spectra, and their non-harmonic spectra can inform us about representation-dependent or topological space-dependent  properties. Topological Dirac shares  similar properties with topological Laplacians and have many important applications. 
For knots and links, we formulate a  Khovanov  Laplacian and   Khovanov  Dirac for the graded chain complex that is used to calculate Khovanov homology, thus recovering Khovanov homology $H_{\mathbb{R}}^{r,q}(K)$ via $\ker\Delta_K^{r,q}$. We calculate the spectra of the Khovanov Laplacian for several knots, and find a correlation between the symmetry of these eigenvalues and the chirality of the knots in some cases where Khovanov homology does not detect it. Since   Khovanov Laplacians inform us of representation-dependent properties, it is important to first define a Khovanov Laplacian on the chain complexes as originally formulated, but this is extremely computationally demanding. There are more efficient methods to compute Khovanov homology \cite{Bar-Natan_Fast}, which would lead to alternative  Khovanov Laplacians that still retain Khovanov homology via their kernel, but would likely have different non-harmonic spectra from the one constructed in this work. 

This is the first Laplacian or Dirac constructed for a knot homology theory.  Khovanov  Laplacians and Dirac may be formulated for other knot homology theories to study those. Just as Khovanov homology categorifies the Jones polynomial, there are homology theories, such as Knot Floer homology \cite{OZSVATH200458}, which categorifies the Alexander polynomial, another major link polynomial invariant. Connections may be possible between a  Khovanov Laplacian designed for Knot Floer homology and the graph Laplacian introduced in \cite{silver_knot_2019}, which can also be used to recover the Alexander polynomial. This  Khovanov Laplacian could be extended to a persistent topological Laplacian by incorporating ideas from the theory of evolutionary Khovanov homology \cite{Shen_2024}. Finally, extension to  persistent Khovanov Dirac can be formulated via the approaches on persistent Dirac \cite{ameneyro2024quantum, wee2023persistent, suwayyid2023persistent}.

\section*{Code availability}

Code used in this work is available at  \href{https://github.com/bdjones13/KhovanovLaplacian/}{https://github.com/bdjones13/KhovanovLaplacian/}.

\section*{Acknowledgments}
This work was supported in part by NIH grants R01AI164266 and R35GM148196, NSF grants DMS-2052983   and IIS-1900473,   MSU Foundation, and Bristol-Myers Squibb 65109.

\bibliographystyle{abbrv}

\bibliography{main_v3}

\end{document}